\documentclass[12pt,final]{article}
\usepackage[margin=2.5cm,top=1.5cm]{geometry}
\usepackage[all,arc]{xy}
\usepackage{mathrsfs}
\usepackage{tikz}
\usepackage{graphicx}
\usepackage{mathtools}
 \usepackage{subfigure}
\usepackage[centerlast]{caption2}
\usepackage{amsmath,amsthm,amssymb,enumerate}
\usepackage{color}
\usepackage{setspace}

\newtheorem{defn}{Definition}[section]
\newtheorem{lem}[defn]{Lemma}
\newtheorem{theo}[defn]{Theorem}
\newtheorem{cor}[defn]{Corollary}

\newtheorem{prob}[defn]{Question}

\numberwithin{equation}{section}

\newcommand\keywordsname{Key words}
\newcommand\AMSname{AMS subject classifications}

\newenvironment{@abssec}[1]{%
     \if@twocolumn
       \section*{#1}%
     \else
       \vspace{.05in}\footnotesize
       \parindent .2in
         {\upshape\bfseries #1. }\ignorespaces
     \fi}
     {\if@twocolumn\else\par\vspace{.1in}\fi}

\begin{document}

\title{A Sharp upper bound for the spectral radius of a nonnegative  matrix and applications\footnote{L. You's research is supported by the Zhujiang Technology New Star Foundation of Guangzhou (Grant No. 2011J2200090) and Program on International Cooperation and Innovation,
Department of Education, Guangdong Province (Grant No. 2012gjhz0007),
X.-D Zhang's research is supported by the NNSF of China (Grant No. 11271256).}}

\author{Lihua You$^{a,}$\footnote{{\it{Email address:\;}}ylhua@scnu.edu.cn}
\qquad Yujie Shu$^{a,}$\footnote{{\it{Email address:\;}}1020697000@qq.com. }
 \qquad Xiao-Dong Zhang$^{b,}$\footnote{{\it{Corresponding author:\;}}xiaodong@sjtu.edu.cn.}
 }\vskip.2cm
\date{{\small
$^{a}$ School of Mathematical Sciences, South China Normal University, Guangzhou, 510631, P.R. China\\
$^{b,}$ Department of Mathematics, Shanghai Jiao Tong University,  Shanghai, 200240, P.R. China\\
{\bf Dedicated to Professor Miroslav Fiedler  on the occasion of his 90th birthday}
}}

\maketitle

\noindent {\bf Abstract }
In this paper, we obtain
 a sharp upper bound for the spectral radius of  a nonnegative matrix. This result is
 used to  present upper bounds for the adjacency  spectral radius, the Laplacian spectral radius,  the signless Laplacian spectral radius, the distance spectral radius,
 the distance Laplacian spectral radius, the distance signless Laplacian spectral radius of  a graph or a digraph. These results are new or generalize some known results.



\section{Introduction}

\hskip.6cm
We begin by recalling some definitions.
Let $M$ be an $n\times n$  matrix, $\lambda_1, \lambda_2, \ldots, \lambda_n$ be the eigenvalues of $M$.
 It is obvious that the eigenvalues may be complex numbers since $M$ is not symmetric in general.
 We usually assume that $|\lambda_1|\geq |\lambda_2|\geq\ldots \geq |\lambda_n|$.
 The spectral radius of $M$ is defined as $\rho(M)=|\lambda_1|$, i.e., it is the largest modulus
 of the eigenvalues of $M$.   If $M$ is a nonnegative matrix, it follows from the Perron-Frobenius theorem that
 the spectral radius $\rho(M)$ is a eigenvalue of $M$.
 If $M$ is a nonnegative irreducible  matrix, it follows from the Perron-Frobenius theorem that
 $\rho(M)=\lambda_1$ is simple.

Let $G=(V,E)$ be a simple graph with vertex set $V=V(G)=\{v_1,v_2,\ldots,v_n\}$ and edge set $E=E(G)$.
 The Laplacian matrix and the signless Laplacian matrix of $G$ are defined as
 $$L(G)=diag(G)-A(G), \hskip.3cm Q(G)=diag(G)+A(G),$$ respectively,
 where $A(G)=(a_{ij})$ is the adjacency matrix of $G$, $diag(G)=diag(d_1, d_2, \ldots, d_n)$ is the diagonal matrix of vertex degrees of $G$
 and $d_i$ is the degree of vertex $v_i$.
 The spectral radius of $A(G)$, $L(G)$ and $Q(G)$,  denoted by $\rho(G), \mu(G) $ and $q(G)$,
 are called the (adjacency) spectral radius of $G$, the   Laplacian spectral radius of $G$, and  the  signless Laplacian spectral radius of $G$, respectively.  In 1973, Fiedler \cite{fiedler1973}  studied the Lapacian spectra, in particular, the second small eigenvalue which is called algebra connectivity.  Since then, the Laplacian matrix have been extensively investigated. Further, Fiedler \cite{fiedler1991} gave an excellent survey for the Laplacian matrix.

Let $G=(V,E)$ be a connected graph with vertex set $V=V(G)=\{v_1,v_2,\ldots,v_n\}$ and edge set $E=E(G)$.
For $u,v\in V(G)$, the distance between $u$ and $v$, denoted by $d_G(u,v)$,
is the length of the shortest path connecting them in $G$.
The distance matrix of $G$ is the $n\times n$ matrix $\mathcal{D}(G)=(d_{ij})$ where $d_{ij}=d_G(v_i,v_j)$.
In fact, for $1\leq i\leq n$, the transmission of vertex $v_i$, $Tr_G(v_i)$ is just the $i$-th row sum
 of $\mathcal{D}(G)$.  So for convenience,  we also call $Tr_G(v_i)$  the distance degree of vertex $v_i$ in $G$,
 denoted by  $D_i$, that is, $D_i=\sum\limits_{j=1}^nd_{ij}=Tr_G(v_i)$.

 Let $Tr(G)=diag(D_1, D_2, \ldots, D_n)$ be the diagonal matrix of vertex transmissions of $G$.
  The distance Laplacian matrix and the distance signless Laplacian matrix of $G$ are
   the $n\times n$ matrix defined by Aouchiche and Hansen as \cite{2013A}
 $$\mathcal{L}(G)=Tr(G)-\mathcal{D}(G), \hskip.3cm \mathcal{Q}(G)=Tr(G)+\mathcal{D}(G).$$

 \noindent The spectral radius of $\mathcal{D}(G),$  $\mathcal{L}(G)$ and $\mathcal{Q}(G)$,  denoted by $\rho^{\mathcal{D}}(G)$,
 $\mu^{\mathcal{D}}(G)$ and  $q^{\mathcal{D}}(G)$, are called the distance spectral radius of $G$,
 the distance  Laplacian spectral radius of $G$, and  the distance signless Laplacian spectral radius of $G$, respectively.

Let $\overrightarrow{G}=(V, E)$ be a digraph,
where $V=V(\overrightarrow{G})=\{v_1, v_2,
\ldots, v_n \}$ and $E=E(\overrightarrow{G})$ are the  vertex set and arc set of $\overrightarrow{G}$,
respectively. A digraph $\overrightarrow{G}$ is simple if it has no loops and
multiple arcs. A digraph  $\overrightarrow{G}$ is strongly connected if for every
pair of vertices $v_i, v_j\in V(\overrightarrow{G})$, there are  directed paths from $v_i$
to $v_j$ and from $v_j$ to $v_i$. In this paper, we consider finite, simple digraphs.

Let $\overrightarrow{G}$ be a digraph.
Let $N^+_{\overrightarrow{G}}(v_i)=\{v_j\in V(\overrightarrow{G})|$ $(v_i, v_j) \in E(\overrightarrow{G})\}$
denote the set of the out-neighbors of $v_i$,
$d^+_i=|N^+_{\overrightarrow{G}}(v_i)|$ denote the out-degree of the vertex $v_i$ in $\overrightarrow{G}$.

 For a digraph $\overrightarrow{G}$, let $A(\overrightarrow{G})=(a_{ij})$ denote the adjacency matrix of $\overrightarrow{G}$,
 where $a_{ij}$ is equal to the number of arcs $(v_i, v_j)$.
 Let $diag(\overrightarrow{G})=diag(d^+_1, d^+_2, \ldots, d^+_n)$ be the diagonal matrix of vertex out-degrees of $\overrightarrow{G}$ and $$L(\overrightarrow{G})= diag(\overrightarrow{G})-A(\overrightarrow{G}), \hskip.3cm
 Q(\overrightarrow{G})= diag(\overrightarrow{G})+A(\overrightarrow{G})$$
 be the Laplacian matrix of $\overrightarrow{G}$ and the signless Laplacian matrix of $\overrightarrow{G}$, respectively.
 The spectral radius of $A(\overrightarrow{G}), L(\overrightarrow{G})$  and $Q(\overrightarrow{G})$,
  denoted by $\rho(\overrightarrow{G}), \mu(\overrightarrow{G})$ and $q(\overrightarrow{G})$,
  are called the (adjacency) spectral radius of $\overrightarrow{G}$, the Laplacian spectral radius of $\overrightarrow{G}$,
  and the signless Laplacian spectral radius of $\overrightarrow{G}$, respectively.

For $u,v\in V(\overrightarrow{G})$, the distance from $u$ to $v$, denoted by $d_{\overrightarrow{G}}(u,v)$,
is the length of the shortest directed path from $u$ to $v$ in ${\overrightarrow{G}}$.
For $u\in V({\overrightarrow{G}})$, the transmission of vertex $u$ in ${\overrightarrow{G}}$ is the sum of distances
 from $u$ to all other vertices of ${\overrightarrow{G}}$, denoted by $Tr_{{\overrightarrow{G}}}(u)$.

Let ${\overrightarrow{G}}$ be a strong connected digraph with vertex set $V({\overrightarrow{G}})=\{v_1, v_2, \ldots, v_n\}$.
The distance matrix of ${\overrightarrow{G}}$ is the $n\times n$ matrix $\mathcal{D}({\overrightarrow{G}})=(d_{ij})$
where $d_{ij}=d_{\overrightarrow{G}}(v_i,v_j)$.
In fact, for $1\leq i\leq n$, the transmission of vertex $v_i$, $Tr_{\overrightarrow{G}}(v_i)$ is just the $i$-th row sum
 of $\mathcal{D}(\overrightarrow{G})$.  So for convenience,
 we also call $Tr_{\overrightarrow{G}}(v_i)$  the distance degree of vertex $v_i$ in $\overrightarrow{G}$,
 denoted by  $D_i^+$, that is, $D_i^+=\sum\limits_{j=1}^nd_{ij}=Tr_{\overrightarrow{G}}(v_i)$.

 Let  $Tr(\overrightarrow{G})=diag(D_1^+, D_2^+, \ldots, D_n^+)$
 be the diagonal matrix of vertex transmissions of $\overrightarrow{G}$.
 The distance Laplacian matrix and the distance signless Laplacian matrix of ${\overrightarrow{G}}$ are the $n\times n$ matrices defined similar to the undirected graph by Aouchiche and Hansen as (\cite{2013A})
 $$\mathcal{L}({\overrightarrow{G}})=Tr({\overrightarrow{G}})-\mathcal{D}({\overrightarrow{G}}), \hskip.3cm
 \mathcal{Q}({\overrightarrow{G}})=Tr({\overrightarrow{G}})+\mathcal{D}({\overrightarrow{G}}).$$
 \noindent The   spectral radius of $\mathcal{D}({\overrightarrow{G}})$, $\mathcal{L}({\overrightarrow{G}})$ and
 $\mathcal{Q}({\overrightarrow{G}})$, denoted by $\rho^{\mathcal{D}}(\overrightarrow{G})$,
 $\mu^{\mathcal{D}}(\overrightarrow{G})$ and $q^{\mathcal{D}}(\overrightarrow{G})$,
 are called the distance  spectral radius of $\overrightarrow{G}$,
 the distance Laplacian spectral radius of $\overrightarrow{G}$
 and the distance signless Laplacian spectral radius of $\overrightarrow{G}$, respectively.

Let $G=(V,E)$ be a graph, for $v_i, v_j\in V$, if $v_i$ is adjacent to $v_j$,  we denote it by $i\sim j$.
Moreover,  we call $m_i=\frac{\sum\limits_{i\sim j}{d_j}}{d_i}$ the average degree of the neighbors of $v_i$.
Let $\overrightarrow{G}=(V,E)$ be a digraph, for $v_i, v_j\in V$, if arc $(v_i, v_j)\in E$, we denote it by $i\sim j$.
Moreover, we call $m^{+}_i=\frac{\sum\limits_{i\sim j}{d^{+}_j}}{d^{+}_i}$ the average out-degree of the out-neighbors of $v_i$,
where $d^+_i$ is the out-degree of vertex $v_i$ in $\overrightarrow{G}$.

A regular graph is a graph where every vertex has the same degree.
A {\it bipartite semi-regular}  graph is a bipartite graph $G=(U,V,E)$ for which every two vertices on the same side of the given bipartition have the same degree as each other.

So far, there are many results on the bounds of the  spectral radius of a matrix and a nonnegative matrix, the  spectral radius,
the Laplacian spectral radius, the signless Laplacian spectral radius, the distance spectral radius,
the distance Laplacian spectral radius and the distance signless Laplacian spectral radius of a graph and a digraph, see [1,3--13,15--17].
The following  are some results on the above spectral radii of graphs and digraphs in terms of
degree, average degree, out-degree and so on. 

\vskip.1cm
$\rho(G)\leq \max \limits_{1\leq i\leq n}\big\{\sqrt{d_im_i}\big\}  $ \hskip11.2cm (1.1)  

\vskip.1cm
$ \mu(G)\leq \max\limits_{1\leq i\leq n}\big\{d_i+\sqrt{d_im_i}\big\} $ \hskip10.5cm (1.2)

\vskip.1cm
$ q(G)\leq \max\limits_{1\leq i\leq n}\big\{d_i+\sqrt{d_im_i}\big\} $ \hskip10.55cm (1.3) 

\vskip.2cm
$ q(\overrightarrow{G})\leq \max\limits_{1\leq i\leq n} \bigg\{d^+_i+\sqrt{\sum\limits_{j\sim i}{d^+_j}}\bigg\}$ \hskip9.96cm (1.4) 

\vskip.2cm
We can see that there are few results about the  distance Laplacian spectral radius of $G$,
the Laplacian spectral radius of $\overrightarrow{G}$, the distance Laplacian spectral radius of $\overrightarrow{G}$
and the distance signless Laplacian spectral radius of $\overrightarrow{G}$.
Maybe  one reason  is the  Laplacian matrix and the  distance Laplacian matrix are not nonnegative matrices.

In this paper, we obtain
 sharp upper  bound for the spectral radius of a matrix or a nonnegative matrix in Section 2,
and then we apply these bounds to various matrices associated with a graph or a digraph,
  obtain some new results or known results about various spectral radii,
  including  the (adjacency)  spectral radius, the Laplacian spectral radius,
  the signless Laplacian spectral radius, the distance spectral radius,
 the distance Laplacian spectral radius, the distance signless Laplacian spectral radius and so on.

\section{Main results}

\hskip.6cm In this section, we will obtain the sharp upper bound for the spectral radius of a (nonnegative) matrix,
The techniques used in this section is motivated by \cite{2010} et al.

\begin{theo}\label{T1}
Let $B=(b_{ij})$ be an  $n\times n$  nonnegative matrix, $l_i$ be the number  of the nonzero entries except for the diagonal entry of the $i$th row for any $i\in \{1,2,\ldots, n\}$, say, $l_i=|\{b_{ij} \hskip.1cm |\hskip.1cm  b_{ij}\not=0, j\in\{1,2,\ldots, n\}\backslash \{i\}\}|$, $X=(x_1,x_2,\dots,x_n)^T$ be the eigenvector of $B$ corresponding to the eigenvalue $\rho(B)$. Then
\begin{equation}\label{eq21}
\rho(B)\leq \max\limits_{1\leq i \leq n}\bigg\{b_{ii}+\sqrt{\sum\limits_{k=1, k\neq i}^{n}{l_kb_{ki}^2}} \bigg\}.
\end{equation}
Moreover, if the equality in (\ref{eq21}) holds, then   $b_{ii}+\sqrt{\sum\limits_{k=1, k\neq i}^{n}{l_{k}b_{ki}^2}}=b_{jj}+\sqrt{\sum\limits_{k=1, k\neq j}^{n}{l_{k}b_{kj}^2}}$ for any $i,j\in \{s\mid x_s \neq 0, 1\leq s\leq n\}$.
Furthermore, if $B$ is  irreducible,
and the equality in (\ref{eq21}) holds, then  $b_{ii}+\sqrt{\sum\limits_{k=1, k\neq i}^{n}{l_{k}b_{ki}^2}}=b_{jj}+\sqrt{\sum\limits_{k=1, k\neq j}^{n}{l_{k}b_{kj}^2}}$ for any $i,j\in \{1,2,\ldots,n\}$.
\end{theo}

\begin{proof}
 For each $i\in \{1,2,\ldots, n\}$, by $BX=\rho(B)X,$  we have
$\rho(B)x_i=\sum\limits_{j=1}^{n}b_{ij}x_j,$ 
then
$$(\rho(B)-b_{ii})x_i=\sum\limits_{j\neq i, b_{ij}\neq0}{b_{ij}x_j},$$
and thus by Cauchy inequality, we have

\begin{equation*}
\begin{aligned}
(\rho(B)-b_{ii})^2x_i^2=(\sum\limits_{j\neq i,b_{ij}\neq0}{b_{ij}x_j})^2\leq l_{i}\sum\limits_{j\neq i, b_{ij}\neq 0}{(b_{ij}x_{j})^2}.
\end{aligned}
\end{equation*}
   Then
\begin{equation*}
\sum\limits_{i=1}^n{[(\rho(B)-b_{ii})x_i]^2}
\leq\sum\limits_{i=1}^n({l_{i}}{\sum\limits_{j\neq i,b_{ij}\neq0}{(b_{ij}x_j})^2})
 =\sum\limits_{i=1}^n(\sum\limits_{j\neq i,b_{ij}\neq0}l_ib_{ij}^2x_j^2)
 =\sum\limits_{i=1}^n[(\sum\limits_{k=1, k\neq i}^n{l_{k}b_{ki}^2})x_{i}^2],
\end{equation*}
thus we have
\begin{equation}\label{eq22}
\sum\limits_{i=1}^n{\big((\rho(B)-b_{ii})^2-\sum\limits_{k=1, k\neq i}^n{l_{k}b_{ki}^2}\big )x_{i}^2}\leq 0.
\end{equation}
Therefore there must exist some  $j\in\{1,2,\ldots, n\}$ such that
\begin{equation*}
(\rho(B)-b_{jj})^2-\sum\limits_{k\neq j}{l_{k}b_{kj}^2}\leq 0,
\end{equation*}
so
\begin{equation*}
\rho(B)\leq b_{jj}+\sqrt{\sum\limits_{k\neq j}{l_{k}b_{kj}^2}}\leq \max\limits_{1\leq i\leq n} \bigg\{b_{ii}+\sqrt{\sum\limits_{k\neq i}{l_{k}b_{ki}^2}}\bigg \}.
\end{equation*}
If $\rho(B)=\max\limits_{1\leq i\leq n} \bigg\{b_{ii}+\sqrt{\sum\limits_{k\neq i}{l_{k}b_{ki}^2}}\bigg\}$, then for any $j\in \{1,2,\ldots, n\}$,
we have $\rho(B)\geq b_{jj}+\sqrt{\sum\limits_{k\neq j}{l_{k}b_{kj}^2}},$
 then
\begin{equation}\label{eq23}
(\rho(B)-b_{jj})^2-\sum\limits_{k\neq j}{l_{k}b_{kj}^2}\geq 0,
\end{equation}
and thus
\begin{equation}\label{eq24}
\sum\limits_{j=1}^n{\big((\rho(B)-b_{jj})^2-\sum\limits_{k\neq j}{l_{k}b_{kj}^2}\big )x_{j}^2}\geq 0.
\end{equation}
Combining (\ref{eq22}) and (\ref{eq24}), it implies that  $$\sum\limits_{i=1}^n{\big((\rho(B)-b_{ii})^2-\sum\limits_{k\neq i}{l_{k}b_{ki}^2}\big )x_{i}^2}=0.$$
Noting that (\ref{eq23}) holds for any  $j\in \{1,2,\ldots, n\}$,
we have $(\rho(B)-b_{ii})^2-\sum\limits_{k\neq i}{l_{k}b_{ki}^2}=0$ for any $i\in \{s\mid x_s \neq0, 1\leq s\leq n\}$,
and thus
$b_{ii}+\sqrt{\sum\limits_{k\neq i}{l_{k}b_{ki}^2}}=b_{jj}+\sqrt{\sum\limits_{k\neq j}{l_{k}b_{kj}^2}}$
for any $i,j\in \{s\mid x_s \neq0, 1\leq s\leq n\}.$

Furthermore, if $B$ is irreducible, then $x_i>0$ for each $i\in \{1,2,\ldots, n\}$ by Perron-Frobenius theorem,
and thus   $b_{ii}+\sqrt{\sum\limits_{k=1, k\neq i}^{n}{l_{k}b_{ki}^2}}=b_{jj}+\sqrt{\sum\limits_{k=1, k\neq j}^{n}{l_{k}b_{kj}^2}}$ for any $i,j\in \{1,2,\ldots, n\}$ if the equality in (\ref{eq21}) holds.
\end{proof}

It is natural that we want to know under what conditions  the equality in (\ref{eq21}) holds.

\begin{prob}
Look for the  necessity and sufficiency conditions of the equality in (\ref{eq21})  holds.
\end{prob}

\begin{lem} (\cite{2005})\label{lem31}
Let $B=(b_{ij})$ be an $n\times n$  nonnegative irreducible matrix and $A=(a_{ij})$ is a complex matrix.
Let $|A|=(|a_{ij}|)$, if $b_{ij}\geq |a_{ij}|$ for any $i,j\in\{1,2,\ldots, n\}$, we denote by  $B\geq|A|$, then  $\rho(B)\geq\rho(A)$.
\end{lem}
By Lemma \ref{lem31} we know that for any connected graph $G$ and any strong connected digraph $\overrightarrow{G}$,
$\mu(G)\leq q(G)$ and $\mu(\overrightarrow{G})\leq q(\overrightarrow{G})$. In fact, we have

\begin{lem} (\cite{1979}) \label{lem32}
Let $G=(V, E)$ be a connected graph on $n$ vertices.
 Then $\mu(G)\leq q(G)$, with equality if and only if $G$ is a bipartite graph.
\end{lem}

\begin{cor}\label{cor25}
Let $A=(a_{ij})$ be an  $n\times n$ complex irreducible matrix, $l_i$ be the number  of the nonzero entries except for the diagonal entry of the $i$th row for any $i\in \{1,2,\ldots, n\}$, say, $l_i=|\{a_{ij} \hskip.1cm |\hskip.1cm  a_{ij}\not=0, j\in\{1,2,\ldots, n\}\backslash \{i\}\}|$. Then
\begin{equation}\label{eq2400}
\rho(A)\leq \max\limits_{1\leq i \leq n}\bigg\{|a_{ii}|+\sqrt{\sum\limits_{k=1, k\neq i}^{n}{l_k|a_{ki}|^2}} \bigg\}.
\end{equation}
If the equality holds, then  $|a_{ii}|+\sqrt{\sum\limits_{k=1, k\neq i}^{n}{l_{k}|a_{ki}|^2}}
=|a_{jj}|+\sqrt{\sum\limits_{k=1, k\neq j}^{n}{l_{k}|a_{kj}|^2}}$ for any $i,j\in \{1,2,\ldots,n\}$.
\end{cor}
\begin{proof}
Let $B=(|a_{ij}|)$, then $B$ is a nonnegative irreducible matrix. Thus  $\rho(A)\leq \rho (B)$ by  Lemma \ref{lem31},  and (\ref{eq2400}) holds by  Theorem \ref{T1}.
\end{proof}

%
%

\section{Various spectral radii of a graph }
\hskip0.6cm  Let $G$ be a  graph,
 the adjacency  matrix $A(G)$, the Laplacian matrix $L(G)$,
  the signless Laplacian matrix $Q(G)$, and the (adjacency)  spectral radius $\rho(G)$, the Laplacian spectral radius $\mu(G)$,
  the signless Laplacian spectral radius $q(G)$ are defined as Section 1.
   Let $G$ be a connected graph, the distance matrix $\mathcal{D}(G)$,
 the distance Laplacian matrix $\mathcal{L}(G)$, the distance signless Laplacian matrix $\mathcal{Q}(G)$,
and  the distance spectral radius $\rho^{\mathcal{D}}(G)$,
 the distance Laplacian spectral radius $\mu^{\mathcal{D}}(G)$, the distance signless Laplacian spectral radius  $q^{\mathcal{D}}(G)$
 are defined as Section 1.
In this section, we will apply Theorems \ref{T1} to $A(G)$,  $Q(G)$, $\mathcal{D}(G)$ and $\mathcal{Q}(G)$,
and apply Corollary \ref{cor25} to $L(G)$ and $\mathcal{L}(G)$,
to obtain some new results or known results on the spectral radius.

\subsection{Adjacency spectral radius of  a graph }

\begin{lem}\label{lem33}
Let $G=(V,E)$ be a simple connected graph with vertex set $V=\{v_1, v_2 \ldots, v_n\}$. For any $v_i\in V$, the degree of $v_i$ and the average degree of the vertices adjacent to $v_i$ are denoted by $d_i$ and $m_i$, respectively. Then $d_1m_1=d_2m_2=\ldots =d_nm_n$ holds if and only if $G$ is a regular graph or a bipartite semi-regular graph.
\end{lem}
\begin{proof}
If $G$ is a regular graph or a bipartite semi-regular graph, we can check that $d_1m_1=d_2m_2=\ldots =d_nm_n$ holds immediately.

Conversely, let $d_1m_1=d_2m_2=\ldots =d_nm_n$ holds. Now we show that $G$ is a regular or a bipartite semi-regular graph.

Let vertex $v_n$  be the lowest degree vertex in $G$,  say, $d_n=\min\{d_i\hskip.1cm |\hskip.1cm 1\leq i\leq n\}$.
Let $d_n=r$, and  the neighbors of $v_n$ be $v_{i_1},v_{i_2},\ldots, v_{i_r}$.
Let $d_{i_1}=\max\{d_{i_j}\hskip.1cm |\hskip.1cm 1\leq j\leq r\}$, denoted by $s=d_{i_1}$.
Then $m_n\leq s$ by $m_i=\frac{\sum\limits_{i\sim j}d_j}{d_i}$ and thus $d_nm_n=rm_n\leq rs$.

On the other hand, for vertex $v_{i_1}$, we have $d_{i_1}m_{i_1}=sm_{i_1}\geq rs$,
then  $rs\leq d_{i_1}m_{i_1}=d_nm_n\leq rs$, thus $d_{i_1}m_{i_1}=d_nm_n=rs$, it implies $m_n= s$ and $m_{i_1}=r$.
Therefore  by the definitions of $s$ and $r$, we know $v_{i_1},v_{i_2},\dots,v_{i_r}$ must have the same degree, say, $s=d_{i_1}=d_{i_2}=\ldots=d_{i_r}$,
and all the neighbors of $v_{i_1}$ must have the same degree  $r$.

Similar to the above arguments, we can show that  the vertices with degree $r$  are  adjacent to the vertices with degree $s$,
 and the vertices with degree $s$  are  adjacent to the vertices with degree $r$  in $G$.

Now we assume that $G$ is not bipartite. Then $G$ has at least an odd cycle.
Let $C=v_{j_1}v_{j_2}\ldots v_{j_{2k-1}}v_{j_{2k}}v_{j_{2k+1}}v_{j_1}$ be an odd cycle of length $2k+1$ in $G$.
Clearly,  the degree of the vertex $v_{j_1}$, say $d_{j_1}$, is either $r$ or $s$.
Without loss of generality, we assume that $d_{j_1}=r$. Since the vertices with degree $r$ are adjacent to the vertices with degree $s$,
and the vertices with degree $s$  are  adjacent to the vertices with degree $r$, then
 $d_{j_2}=s$,  $d_{j_3}=r$, \ldots,  $d_{j_{2k-1}}=r$,  $d_{j_{2k}}=s$, $d_{j_{2k+1}}=r$ and $d_{j_{1}}=s$
 by $v_{j_{2k+1}}$ and $v_{j_1}$ are adjacent, thus $r=s$, it implies that $G$ is regular.

Hence the graph $G$ is a regular graph or a bipartite semi-regular graph.
\end{proof}

\begin{theo}\label{theo34}(\cite{1998}, Theorem 1.)
 Let $G=(V,E)$ be a simple graph on $n$ vertices. Then
$\rho(G)\leq \max\limits_{1\leq i\leq n}\sqrt{d_{i}m_{i}}.$
Moreover, if $G$ is a connected graph,  then the equality holds if and only if $G$ is a regular or bipartite semi-regular graph.
\end{theo}

\begin{proof}
We apply
 Theorem \ref{T1} to $A(G)$. Since $b_{ii}=0$,
$b_{ij}=\left\{\begin{array}{cc}
                                                                     1, & \mbox{ if } v_i \mbox{ and } v_j \mbox{ are adjacent,} \\
                                                                     0, & \mbox{ otherwise, }
                                                                   \end{array}\right.$
 and $l_i=d_i$ for $i=1,2,\ldots,n$, then $b_{ii}+\sqrt{\sum\limits_{k\neq i}{l_{k}b_{ki}^2}}=\sqrt{d_im_i}$ for $i=1,2,\ldots,n$,
 and thus $\rho(G)\leq \max\limits_{1\leq i\leq n}\sqrt{d_{i}m_{i}}$ by (\ref{eq21}).

If $G$ is a connected graph, now we show  the equality holds if and only if $G$ is a regular or bipartite semi-regular graph.

  If $G$ is a connected graph and $\rho(G)= \max\limits_{1\leq i\leq n}\sqrt{d_{i}m_{i}}$,
  then  $\sqrt{d_1m_1}=\sqrt{d_2m_2}=\ldots=\sqrt{d_nm_n}$ by Theorem \ref{T1}, and thus $d_1m_1=d_2m_2=\ldots=d_nm_n$.
  Therefore,  $G$ is a regular or bipartite semi-regular graph by Lemma \ref{lem33}.

 On the other hand,  if $G$ is connected and $G$ is a regular or bipartite semi-regular graph, then $d_1m_1=d_2m_2=\ldots=d_nm_n$ by Lemma \ref{lem33},
  thus $\sqrt{d_1m_1}=\sqrt{d_2m_2}=\ldots=\sqrt{d_nm_n}$,
  and $\rho(G)\leq \max\limits_{1\leq i\leq n}\sqrt{d_{i}m_{i}}=\sqrt{d_1m_1}$.
Then we complete the proof by the following two cases.

{\bf Case 1: }  $G$ is a regular graph with degree $r$.

It is well known that $r=\sqrt{d_1m_1}$ is a eigenvalue of $G$, so $\sqrt{d_1m_1}\leq \rho(G)$.
Thus  $\rho(G)=\max\limits_{1\leq i\leq n}\sqrt{d_{i}m_{i}}=r$ by $\rho(G)\leq\sqrt{d_1m_1}=r$.

{\bf Case 2: }  $G$ is a bipartite semi-regular graph.

We assume that  the two  bipartition of $G$ with degree $r$ and $s$, respectively.
It is easy to check that $\sqrt{rs}=\sqrt{d_1m_1}$ is a eigenvalue of $G$, so $\sqrt{rs}\leq \rho(G)$.
Thus  $\rho(G)=\max\limits_{1\leq i\leq n}\sqrt{d_{i}m_{i}}=\sqrt{rs}$ by $\rho(G)\leq\sqrt{d_1m_1}=\sqrt{rs}$.
\end{proof}

\subsection{(Signless) Laplacian spectral radius of   a graph }
\begin{lem}(\cite{2005}) \label{lem23}
Let $A$ be a nonnegative matrix with the spectral radius $\rho(A)$ and the row sum $r_1,r_2,\ldots,r_n$. Then
$\min\limits_{1\leq i\leq n}r_i\leq\rho(A)\leq \max\limits_{1\leq i\leq n}r_i.$
Moreover, if $A$ is a irreducible matrix, the one of equalities holds if and only if the row sums of $A$ are all equal.
\end{lem}

\begin{lem}\label{lem37}
Let $G=(V,E)$ be a  simple connected graph with vertex set $V=\{v_1,v_2,\ldots,v_n\}$. For any $v_i\in V$, the degree of $v_i$ and the average degree of the vertices adjacent to $v_i$ are denoted by $d_i$ and $m_i$, respectively. Then $d_1+\sqrt{d_1m_1}=d_2+\sqrt{d_2m_2}=\ldots =d_n+\sqrt{d_nm_n}$ holds if and only if $G$ is a regular graph.
\end{lem}

\begin{proof}
If $G$ is a regular graph, we can check  $d_1+\sqrt{d_1m_1}=d_2+\sqrt{d_2m_2}=\ldots =d_n+\sqrt{d_nm_n}$ holds immediately.

Conversely, let $d_1+\sqrt{d_1m_1}=d_2+\sqrt{d_2m_2}=\ldots =d_n+\sqrt{d_nm_n}$ holds. Now we show $G$ is a regular graph.

Let vertex $v_n$  be the lowest degree vertex in $G$,  say, $d_n=\min\{d_i\hskip.1cm |\hskip.1cm 1\leq i\leq n\}$.
Let $d_n=r$, and  the neighbors of $v_n$ be $v_{i_1},v_{i_2},\ldots, v_{i_r}$.
Let $d_{i_1}=\max\{d_{i_j}\hskip.1cm |\hskip.1cm 1\leq j\leq r\}$, denoted by $s=d_{i_1}$.
It is obvious that $r\leq s$, $m_n\leq s$ by $m_i=\frac{\sum\limits_{i\sim j}d_j}{d_i}$, and 
thus $d_n+\sqrt{d_nm_n}\leq r+\sqrt{rs}.$

On the other hand, for vertex $v_{i_1}$, we have $d_{i_1}m_{i_1}=sm_{i_1}\geq rs$,
then  $d_{i_1}+\sqrt{d_{i_1}m_{i_1}}\geq s+\sqrt{rs}$.
Thus $s+\sqrt{rs}\leq d_{i_1}+\sqrt{d_{i_1}m_{i_1}}=d_n+\sqrt{d_nm_n}\leq r+\sqrt{rs}\leq s+\sqrt{rs}$,
it implies $m_n= s$, $m_{i_1}=r$ and $r=s$.
Therefore  by the definitions of $s$ and $r$, we know $v_{i_1},v_{i_2},\dots,v_{i_r}$ must have the same degree, say, $s=d_{i_1}=d_{i_2}=\ldots=d_{i_r}$,
and all the neighbors of $v_{i_1}$ must have the same degree  $r(=s)$.

Similar to the above arguments, we can show that  the vertices with degree $r$  are  adjacent to the vertices with degree $s$,
 and the vertices with degree $s$  are  adjacent to the vertices with degree $r$  in $G$. Then $G$ is a regular graph by $r=s$.
\end{proof}

\begin{theo}\label{thm38} Let $G=(V,E)$ be a simple graph on $n$ vertices. Then

(i)\cite{2010} $q(G)\leq \max\limits_{1\leq i\leq n}\big\{d_{i}+\sqrt{d_{i}m_{i}}\big\}.$
Moreover, if $G$ is a connected graph, the equality holds  if and only if $G$ is a regular graph.

(ii)\cite{2004Zhang} If $G$ is a connected graph, then $\mu(G)\leq \max\limits_{1\leq i\leq n}\big\{d_{i}+\sqrt{d_{i}m_{i}}\big\},$
 and  the equality holds if and only if  $G$ is a bipartite regular graph.
\end{theo}

\begin{proof}
Firstly, we show (i) holds.

We apply  Theorem \ref{T1} to $Q(G)$. Since $b_{ii}=d_i$,
                                        $b_{ij}=\left\{\begin{array}{cc}
                                                                     1, & \mbox{ if } v_i \mbox{ and } v_j \mbox{ are adjacent;} \\
                                                                     0, & \mbox{ otherwise, }
                                                                   \end{array}\right.$
$l_i=d_i$ for  $i=1,2,\ldots,n$,  then we have $b_{ii}+\sqrt{\sum\limits_{k\neq i}{l_{k}b_{ki}^2}}=d_i+\sqrt{d_im_i}$ for $i=1,2,\ldots,n$, and thus $q(G)\leq \max\limits_{1\leq i\leq n}\big\{d_{i}+\sqrt{d_{i}m_{i}}\big\}$ by (\ref{eq21}).

Now we show if $G$ is a connected graph,  then the equality holds if and only if $G$ is regular.

If $G$ is a connected graph and $q(G)=\max\limits_{1\leq i\leq n}\big\{d_{i}+\sqrt{d_{i}m_{i}}\big\}$, then  $d_1+\sqrt{d_1m_1}=d_2+\sqrt{d_2m_2}=\ldots=d_n+\sqrt{d_nm_n}$ by Theorem \ref{T1},
and thus $G$ is a regular graph by Lemma \ref{lem37}.

On the other hand,  if $G$ is connected and $G$ is a regular  graph with degree $r$, then $d_1+\sqrt{d_1m_1}=d_2+\sqrt{d_2m_2}=\ldots=d_n+\sqrt{d_nm_n}=2r$  by Lemma \ref{lem37} and
$\max\limits_{1\leq i\leq n}\big\{d_{i}+\sqrt{d_{i}m_{i}}\}=2r$.
It is well known that $q(G)=2r$  by Lemma \ref{lem23},
so  $q(G)=\max\limits_{1\leq i\leq n}\big\{d_{i}+\sqrt{d_{i}m_{i}}\big\}$.

 Similar to the proof of (i), by  Corollary \ref{cor25}, Lemma \ref{lem32} and  the result of (i), we can show (ii) immediately, so we omit it.
\end{proof}





\subsection{Distance spectral radius of a graph }

\begin{theo}\label{thm314}
Let $G=(V,E)$ be a connected graph on $n$ vertices. Then
\begin{equation}\label{eq33}
\begin{aligned}
&\rho^D(G)\leq \max\limits_{1\leq i\leq n}\sqrt{(n-1)\sum\limits_{k=1}^nd^2_{ki}}.
\end{aligned}
\end{equation}
If the equality holds, then $\sum\limits_{k=1}^nd^2_{ki}=\sum\limits_{k=1}^nd^2_{kj}$  for any $i,j\in\{1,2,\ldots,n\}$.\\
\end{theo}

\begin{proof}
We apply  Theorem \ref{T1} to $\mathcal{D}(G)$. Since $b_{ii}=d_{ii}=0$,
$b_{ij}=d_{ij}\not=0$ when $i\not=j$
 and $l_i=n-1$ for $i=1,2,\ldots,n$,
 then $b_{ii}+\sqrt{\sum\limits_{k\neq i}{l_{k}b_{ki}^2}}
 =\sqrt{(n-1)\sum\limits_{k=1}^{n}{d^2_{ki}}}$ for $i=1,2,\ldots,n$,
 and thus (\ref{eq33}) holds
 by (\ref{eq21}).

It is obvious that if  the equality holds, then $\sum\limits_{k=1}^nd^2_{ki}=\sum\limits_{k=1}^nd^2_{kj}$ for any $i,j=1,2,\ldots,n$ by Theorem \ref{T1}.
\end{proof}


\subsection{Distance (signless) Laplacian spectral radius of a graph}

\begin{theo}\label{}
Let $G=(V,E)$ be a connected graph on $n$ vertices. Then
\begin{equation}\label{eq35}
q^D(G)\leq \max\limits_{1\leq i\leq n}\bigg\{D_i+\sqrt{(n-1)\sum\limits_{k=1}^nd^2_{ki}}\bigg\},
\end{equation}
and \begin{equation}\label{eq36}
\mu^D(G)\leq \max\limits_{1\leq i\leq n}\bigg\{D_i+\sqrt{(n-1)\sum\limits_{k=1}^nd^2_{ki}}\bigg\}.
\end{equation}
Moreover, if the equality in  (\ref{eq35}) (or (\ref{eq36})) holds, then $D_i+\sqrt{(n-1)\sum\limits_{k=1}^nd^2_{ki}}=D_j+\sqrt{(n-1)\sum\limits_{k=1}^nd^2_{kj}}$ for any $i,j\in\{1,2,\ldots,n\}$.
\end{theo}

\begin{proof}
We apply Theorem \ref{T1} to $\mathcal{Q}(G)$. Since $b_{ii}=D_i$, $b_{ij}=d_{ij}$ where $i\not= j$,
 and $l_i=n-1$ for $i=1,2,\ldots,n$, then $b_{ii}+\sqrt{\sum\limits_{k\neq i}{l_{k}b_{ki}^2}}=D_i+\sqrt{(n-1)\sum\limits_{k=1}^n{d^2_{ki}}}$ for $i=1,2,\ldots,n$, and thus (\ref{eq35})  holds.

Similarly, we  apply Corollary \ref{cor25} to $\mathcal{L}(G)$ and (\ref{eq36}) holds.
\end{proof}

\section{Various spectral radii of a  digraph }
\hskip0.6cm  Let $\overrightarrow{G}$ be a connected digraph,
 the adjacency  matrix $A(\overrightarrow{G})$, the Laplacian matrix $L(\overrightarrow{G})$,
  the signless Laplacian matrix $Q(\overrightarrow{G})$,  and the (adjacency)  spectral radius $\rho(\overrightarrow{G})$,
  the Laplacian spectral radius $\mu(\overrightarrow{G})$, the signless Laplacian spectral radius $q(\overrightarrow{G})$
   are defined as Section 1.
  Let $\overrightarrow{G}$ be a connected digraph, the distance matrix $\mathcal{D}(\overrightarrow{G})$,
 the distance Laplacian matrix $\mathcal{L}(\overrightarrow{G})$, the distance signless Laplacian matrix $\mathcal{Q}(\overrightarrow{G})$,
and the distance spectral radius $\rho^{\mathcal{D}}(\overrightarrow{G})$,
 the distance Laplacian spectral radius $\mu^{\mathcal{D}}(\overrightarrow{G})$, the distance signless Laplacian spectral radius  $q^{\mathcal{D}}(\overrightarrow{G})$
 are defined as Section 1.
In this section, we will apply Theorems \ref{T1}  to $A(\overrightarrow{G})$,  $Q(\overrightarrow{G})$, $\mathcal{D}(\overrightarrow{G})$
 and $\mathcal{Q}(\overrightarrow{G})$, and apply Corollary \ref{cor25} to
 $L(\overrightarrow{G})$  and $\mathcal{L}(\overrightarrow{G})$,
to obtain some new results or known results on the spectral radius.

\subsection{Adjacency spectral radius of a digraph}

\begin{theo}(\cite{Zhang2002}, Corollary 3.2)\label{thm41}
Let $\overrightarrow{G}=(V,E)$ be a digraph on $n$ vertices. Then
\begin{equation*}
\begin{aligned}
&\rho(\overrightarrow{G})\leq \max\limits_{1\leq i\leq n}\sqrt{\sum\limits_{k\sim i}d_k^+}.
\end{aligned}
\end{equation*}
If  $\overrightarrow{G}$ is connected and the the equality holds, then
$\sum\limits_{k\sim 1}d_k^+=\sum\limits_{k\sim 2}d_k^+=\ldots=\sum\limits_{k\sim n}d_k^+$.
\end{theo}

\begin{proof}
We apply Theorem \ref{T1} to $A(\overrightarrow{G})$. Since $b_{ii}=0$, $b_{ij}=\left\{\begin{array}{cc}
                                                                            1, & \mbox{ if } (v_i, v_j )\in E;\\
                                                                            0, & \mbox{ otherwise, }
                                                                   \end{array}\right.$
 and $l_i=d^+_i$ for $i=1,2,\ldots,n$, then $b_{ii}+\sqrt{\sum\limits_{k\neq i}{l_{k}b_{ki}^2}}
 =\sqrt{\sum\limits_{k\sim i}d_k^+}$, 
 and thus $\rho(\overrightarrow{G})\leq \max\limits_{1\leq i\leq n}\sqrt{\sum\limits_{k\sim i}d_k^+}$ by (\ref{eq21}).

 It is obvious that if  $\overrightarrow{G}$ is connected and the the equality holds, then
$\sum\limits_{k\sim 1}d_k^+=\sum\limits_{k\sim 2}d_k^+=\ldots=\sum\limits_{k\sim n}d_k^+$ by  Theorem \ref{T1}.
\end{proof}

\subsection{(Signless) Laplacian spectral radius of  a digraph}

\begin{theo}\label{thm43}
Let $\overrightarrow{G}=(V,E)$ be a digraph on $n$ vertices. Then

(i) (\cite{2013}, Theorem 3.3) $q(\overrightarrow{G})\leq \max\limits_{1\leq i\leq n}\big\{d_{i}^{+}+\sqrt{\sum\limits_{j\sim i}d_j^+}\big\}.$
Moreover, if $\overrightarrow{G}$ is connected and  the equality  holds, then
$d_{1}^{+}+\sqrt{\sum\limits_{j\sim 1}d_j^+}=d_{2}^{+}+\sqrt{\sum\limits_{j\sim 2}d_j^+}=\ldots=d_{n}^{+}+\sqrt{\sum\limits_{j\sim n}d_j^+}$.

(ii) If $\overrightarrow{G}$ is connected, then $\mu(\overrightarrow{G})\leq \max\limits_{1\leq i\leq n}\big\{d_{i}^{+}+\sqrt{\sum\limits_{j\sim i}d_j^+}\big\},$
and if the equality  holds, then
$d_{1}^{+}+\sqrt{\sum\limits_{j\sim 1}d_j^+}=d_{2}^{+}+\sqrt{\sum\limits_{j\sim 2}d_j^+}=\ldots=d_{n}^{+}+\sqrt{\sum\limits_{j\sim n}d_j^+}$.
\end{theo}

\begin{proof}
We apply Theorem \ref{T1} to $Q(\overrightarrow{G})$.
Since $b_{ii}=d^+_i$, $b_{ij}=\left\{\begin{array}{cc}
                                                                     1, & \mbox{ if } (v_i, v_j )\in E; \\
                                                                     0, & \mbox{ otherwise, }
                                                                   \end{array}\right.$
 and $l_i=d^+_i$ for $i=1,2,\ldots,n$, then $b_{ii}+\sqrt{\sum\limits_{k\neq i}{l_{k}b_{ki}^2}}=d_{i}^{+}+\sqrt{\sum\limits_{j\sim i}d_j^+}$ for $i=1,2,\ldots,n$,
 and thus $q(\overrightarrow{G})\leq \max\limits_{1\leq i\leq n}\big\{d_{i}^{+}+\sqrt{\sum\limits_{j\sim i}d_j^+}\big\}$ by (\ref{eq21}).

 It is obvious that if  $\overrightarrow{G}$ is connected and the equality holds then
 $d_{1}^{+}+\sqrt{\sum\limits_{j\sim 1}d_j^+}=d_{2}^{+}+\sqrt{\sum\limits_{j\sim 2}d_j^+}=\ldots=d_{n}^{+}+\sqrt{\sum\limits_{j\sim n}d_j^+}$
 by Theorem \ref{T1}.

Similar to the proof of (i), we can show (ii) immediately by Corollary \ref{cor25}, so we omit it.
\end{proof}

\subsection{Distance spectral radius of  a digraph }
\begin{theo}\label{thm46}
Let $\overrightarrow{G}=(V,E)$ be a strong connected digraph on $n$ vertices. Then
\begin{equation}\label{eq43}
\begin{aligned}
&\rho^D(\overrightarrow{G})\leq \max\limits_{1\leq i\leq n}\sqrt{(n-1)\sum\limits_{k=1}^n d_{ki}^2}.
\end{aligned}
\end{equation}
If the equality holds, then $\sum\limits_{k=1}^n d^2_{ki}=\sum\limits_{k=1}^n d^2_{kj}$ for any $i,j\in\{1,2,\ldots,n\}$.\\
\end{theo}

\begin{proof}
We apply Theorem \ref{T1} to $\mathcal{D}(\overrightarrow{G})$. Since $b_{ii}=d_{ii}=0$, $b_{ij}=d_{ij}\not=0,$
 and $l_i=n-1$ for any $i=1,2,\ldots,n$, then $b_{ii}+\sqrt{\sum\limits_{k\neq i}{l_{k}b_{ki}^2}}=\sqrt{(n-1)\sum\limits_{k=1}^n{d^2_{ki}}}$ for $i=1,2,\ldots,n$,
 and thus (\ref{eq43}) holds by (\ref{eq21}).

 It is easy that if the equality holds, then $\sum\limits_{k=1}^n d^2_{ki}=\sum\limits_{k=1}^n d^2_{kj}$ for any $i,j\in\{1,2,\ldots,n\}$.
\end{proof}

\subsection{Distance (signless) Laplacian spectral radius of   a digraph}

\begin{theo}\label{thm48}
Let $\overrightarrow{G}=(V,E)$ be a strong connected digraph on $n$ vertices. Then
\begin{equation}\label{eq45}
q^D(\overrightarrow{G})\leq \max\limits_{1\leq i\leq n}\bigg\{D^+_i+\sqrt{(n-1)\sum\limits_{k=1}^n d^2_{ki}}\bigg\},
\end{equation}
and \begin{equation}\label{eq46}
\mu^D(\overrightarrow{G})\leq \max\limits_{1\leq i\leq n}\bigg\{D^+_i+\sqrt{(n-1)\sum\limits_{k=1}^n d^2_{ki}}\bigg\}.
\end{equation}
Moreover, if the equality in (\ref{eq45})  (or (\ref{eq46})) holds then $D^+_i+\sqrt{(n-1)\sum\limits_{k=1}^nd^2_{ki}}=D^+_j+\sqrt{(n-1)\sum\limits_{k=1}^n d^2_{kj}}$  for any $i,j\in\{1,2,\ldots,n\}$.
\end{theo}

\begin{proof}
We apply Theorem \ref{T1} to $\mathcal{Q}(\overrightarrow{G})$. Since $b_{ii}=D^+_i$,
$b_{ij}=d_{ij}\not=0$ for all $i\not=j$, $b_{ii}=d_{ii}=0$,
 and $l_i=n-1$ for $i=1,2,\ldots,n$, then $b_{ii}+\sqrt{\sum\limits_{k\neq i}{l_{k}b_{ki}^2}}=D^+_i+\sqrt{(n-1)\sum\limits_{k=1}^n{d^2_{ki}}}$ for $i=1,2,\ldots,n$,
 and thus (\ref{eq45}) holds by (\ref{eq21}). By Corollary \ref{cor25} and (i), (\ref{eq46}) holds.

It is easy that if the equality in (\ref{eq45})  (or (\ref{eq46})) holds then $D^+_i+\sqrt{(n-1)\sum\limits_{k=1}^nd^2_{ki}}=D^+_j+\sqrt{(n-1)\sum\limits_{k=1}^n d^2_{kj}}$  for any $i,j\in\{1,2,\ldots,n\}$ by Theorem \ref{T1}.
\end{proof}


\begin{thebibliography}{}


\bibitem{2013A}
M. Aouchiche, P. Hansen, Two Laplacian for the distance matrix of a graph, Linear Algebra Appl. 493 (2013) 21--33.

\bibitem{1979}
A. Berman, R.J. Plemmons, Nonnegative Matrices in the Mathematical Sciences, New York, Academic Press, 1979.


\bibitem{2013}
S.B. Bozkurt, D. Bozkurt, On the signless Laplacian spectral radius of digraphs, Ars Combin. 108 (2013) 193--200.

\bibitem{1998}
D.S. Cao, Bounds on eigenvalues and chromatic number, Linear Algebra Appl. 270 (1998) 1--13.

\bibitem{2011Chen}
Y.H. Chen, R.Y. Pan, X.-D. Zhang, Two sharp bounds for the signless Laplacian spectral radius of graphs, Discrete Math Algorithms Appl. 3 (2011) 185--192.

\bibitem{2013Cui}
S.Y. Cui, G.X. Tian, J.J. Guo, A sharp upper bound on the signless Laplacian spectral radius of graphs, Linear Algebra Appl. 439 (2013) 2442--2447.



\bibitem{2003Das}
K.Ch. Das, A characterization on graphs which achieve the upper bounds for the largest Laplacian eigenvalue of graphs, Linear Algebra Appl. 368 (2003) 269--278.


\bibitem{2004Das2}
K.Ch. Das, P. Kumar, Some new  bounds on the spectral radius of graphs, Discrete Math. 281 (2004) 149--161.

\bibitem{fiedler1973}M.~Fiedler, Algebraic connectivity of graphs,  Czechoslovak Math. J. 23(98) (1973), 298--305.
    
    \bibitem{fiedler1991}M.~Fiedler,  A geometric approach to the Laplacian matrix of a graph. Combinatorial and graph-theoretical problems in linear algebra (Minneapolis, MN, 1991), 73¨C98, IMA Vol. Math. Appl., 50, Springer, New York, 1993. 

\bibitem{2005Guo}
J.M. Guo, A new upper bounds for the largest Laplacian eigenvalue of graphs, Linear Algebra Appl. 400 (2005) 61--66.

\bibitem{2013Guo}
J.M. Guo, J.X. Li, W.C. Shiu, A note on the upper bounds for the Laplacian spectral radius of graphs, Linear Algebra Appl. 439 (2013) 1657--1661.

\bibitem{2010}
C.X. He, Y. Liu, Z.H. Zhao, Some new sharp bounds on the distance spectral radius of a graph, MATCH Commum. Math. Comput. Chem. 63 (2010) 783--788.

\bibitem{2005}
R.A. Horn, C.R. Johnson, Matrix Analysis,  Cambridge Univesity Press, 2005.





\bibitem{1998Li}
J.S. Li, X.-D. Zhang, On the Laplacian eigenvalues of a graph, Linear Algebra Appl. 285 (1998) 305--307.

\bibitem{2001Li}
J.S. Li, Y.L. Pan, de Caen's inequality and bounds on the largest Laplacian eigenvalue of a graph, Linear Algebra Appl. 328 (2001) 153--160.

\bibitem{2004}
J.S. Li, Y.L Pan, Upper bounds for the Laplacian graph eigenvalues, Acta Math Sinica, Eng Series. 20 (2004) 803--806.

\bibitem{2013Lin}
H.Q. Lin, J.L. Shu, The distance spectral radius of a digraph, Discrete Applied Math. 161 (2013) 2537--2543.

\bibitem{2013Maden}
A.D. Maden, K.C. Das, A.S. Cevik, Sharp upper bounds on the spectral radius of the signless Laplacian matrix of a graph,
Appl Math Comput. 219 (2013) 5025--5032.

\bibitem{2013Mehnan}
A. Mehnan, Upper and lower bounds for the Perron root of a nonnegative matrix, Linear Multilinear Algebra. 61 (2013) 171--181.

\bibitem{2010}
C.S. Oliveira, L.S.D. Lima, N.M.M.D. Abreu, P. Hansen, Bounds on the index of the signless Laplacian of a graph, Discrete Applied Math. 158 (2010) 355--360.

\bibitem{2004Shu}
J.L. Shu, Y.R. Wu, Sharp upper bounds on the spectral radius of graphs, Linear Algebra Appl. 377 (2004) 241--248.

\bibitem{2009}
G.H. Xu, C.Q. Xu, Sharp bounds for the spectral radius of digraphs, Linear Algebra Appl. 430 (2009) 1607--1612.


\bibitem{2011}
G. Yu, Y. Wu, J. Shu, Sharp bounds on the signless Laplacian spectral radii of graphs, Linear Algebra Appl. 434 (2011) 683--687.




\bibitem{2004Zhang}
X.-D. Zhang, Two sharp upper bounds for Laplacian eigenvalues, Linear Algebra Appl. 376 (2004) 207--213.

\bibitem{Zhang2002}X.-D.~Zhang, J.-S.~Li, Spectral radius of non-negative matrices and digraphs, {\it Acta Math.~Sin.,} 18(2002) 293-300.
\bibitem{2003Zhang}
X.-D. Zhang, R. Luo, The Laplacian eigenvalus of mixed graphs, Linear Algebra Appl. 362 (2003) 109--119.

\end{thebibliography}
\end{document}